\documentclass[11pt, reqno]{amsart} \textwidth=14.9cm \oddsidemargin=1cm \evensidemargin=1cm
\usepackage{amsmath,amssymb,amscd,mathrsfs,stmaryrd,wasysym}
\usepackage{amsmath}
\usepackage{amsxtra}
\usepackage{amscd}
\usepackage{amsthm}
\usepackage{amsfonts}
\usepackage{amssymb}
\usepackage{eucal}
\usepackage{color}
\usepackage{graphicx}%
\usepackage{bm}
\newtheorem{theorem}{Theorem}[section]
\newtheorem{lemma}[theorem]{Lemma}
\newtheorem{proposition}[theorem]{Proposition}

\theoremstyle{definition}
\newtheorem{definition}[theorem]{Definition}
\newtheorem{remark}[theorem]{Remark}

\definecolor{A}{rgb}{.75,1,.75}

\numberwithin{equation}{section}

\begin{document}

\title[Fusion procedure]{Fusion procedure for cyclotomic BMW algebras}
\author[Weideng Cui]{Weideng Cui}
\address{School of Mathematics, Shandong University, Jinan, Shandong 250100, P.R. China.}
\email{cwdeng@amss.ac.cn}

\begin{abstract}
Inspired by the work [IMOg2], in this note, we prove that the pairwise orthogonal primitive idempotents of generic cyclotomic Birman-Murakami-Wenzl algebras can be constructed by consecutive evaluations of a certain rational function. In the appendix, we prove a similar result for generic cyclotomic Nazarov-Wenzl algebras.
\end{abstract}



\maketitle
\medskip
\section{Introduction}
\subsection{}
The primitive idempotents of a symmetric group $\mathfrak{S}_n,$ showed by Jucys [Juc], can be obtained by taking a certain limiting process on a rational function. The process is now commonly known as the fusion procedure, which has been further developed in the situation of Hecke algebras [Ch]; see also [Na2-4]. In [Mo], Molev has presented another approach of the fusion procedure for $\mathfrak{S}_n,$ which depends on the existence of a maximal commutative subalgebra generated by the Jucys-Murphy elements. In his approach, the primitive idempotents are obtained by consecutive evaluations of a certain rational function. The new version of the fusion procedure has been generalized to the Hecke algebras of type $A$ [IMO], to the Brauer algebras [IM, IMOg1], to the Birman-Murakami-Wenzl algebras [IMOg2], to the complex reflection groups of type $G(d,1,n)$ [OgPA1], to the Ariki-Koike algebras [OgPA2], to the wreath products of finite groups by the symmetric group [PA], to the degenerate cyclotomic Hecke algebras [ZL], to the Yokonuma-Hecke algebras [C1], to the cyclotomic Yokonuma-Hecke algebras [C2, Appendix] and to the degenerate cyclotomic Yokonuma-Hecke algebras [C3].

\subsection{}
The Birman-Murakami-Wenzl (for brevity, BMW) algebra was algebraically defined by Birman and Wenzl [BW], and independently Murakami [Mu], which is an algebra generated by some elements satisfying certain particular relations. These relations are in fact implicitly modeled on the ones of certain algebra of tangles studied by Kauffman [Ka] and Morton and Traczyk [MT], which is known as a Kauffman tangle algebra. BMW algebra are closely related to Artin braid groups of type $A,$ Iwahori-Hecke algebras of type $A,$ quantum groups, Brauer algebras and other diagram algebras; see [Eny1-2, HuXi, Hu, LeRa, MW, RuSi4-6, RuSo, Xi] and the references therein.

Motivated by studying link invariants, H\"{a}ring-Oldenburg [HO] introduced a class of finite dimensional associative algebras called cyclotomic Birman-Murakami-Wenzl (for brevity, BMW) algebras, generalizing the notions of BMW algebras. Such algebras are closely related to Artin braid groups of type $B,$ cyclotomic Hecke algebras and other research objects, and have been studied by a lot of authors from different perspectives; see [Go1-4, GoHM1-2, HO, OrRa, RuSi2-3, RuXu, Si, WiYu1-3, Xu, Yu] and so on.

\subsection{}
Inspired by the work [IMOg2] on the fusion procedures of BMW algebras, in this note we prove that a complete set of pairwise orthogonal primitive idempotents of cyclotomic BMW algebras can be derived by consecutive evaluations of a certain rational function in several variables. In the appendix, we prove a similar result for generic cyclotomic Nazarov-Wenzl algebras.

This paper is organized as follows. In Section 2, we recall some preliminaries and introduce the the primitive idempotents $E_{\mathcal{T}}$ of cyclotomic BMW algebras. In Section 3, we establish the fusion formula for the primitive idempotent $E_{\mathcal{T}}.$ In Section 4 (Appendix), we develop the fusion formulas for the primitive idempotents of cyclotomic Nazarov-Wenzl algebras.

\section{Preliminaries}
\subsection{Cyclotomic Birman-Murakami-Wenzl algebras}
\begin{definition}
Assume that $\mathbb{K}$ is an algebraically closed field containing $\delta_{j},$ $0\leq j\leq d-1,$ and some nonzero elements $\rho,$ $q$, $q-q^{-1}$ and $v_i,$ $1\leq i\leq d$, and that they satisfy the relation $\rho-\rho^{-1}=(q-q^{-1})(\delta_{0}-1).$\vskip2mm

Fix $n\geq 1.$ The cyclotomic Birman-Murakami-Wenzl algebra $\mathscr{B}_{d, n}$ is the $\mathbb{K}$-algebra generated by the elements $X_{1}^{\pm 1}, T_{i}^{\pm 1}$ and $E_{i}$ ($1\leq i\leq n-1$) with the following relations:\vskip2mm

(1) (Inverses) $T_{i}T_{i}^{-1}=T_{i}^{-1}T_{i}=1$ and $X_{1}X_{1}^{-1}=X_{1}^{-1}X_{1}=1.$

(2) (Idempotent relations) $E_{i}^{2}=\delta_{0} E_{i}$ for $1\leq i\leq n-1.$

(3) (Affine braid relations)

\hspace{0.7cm}(a) $T_{i}T_{i+1}T_{i}=T_{i+1}T_{i}T_{i+1}$ and $T_{i}T_{j}=T_{j}T_{i}$ if $|i-j|\geq 2.$

\hspace{0.7cm}(b) $X_{1}T_{1}X_{1}T_{1}=T_{1}X_{1}T_{1}X_{1}$ and $X_{1}T_{j}=T_{j}X_{1}$ if $j\geq 2.$

(4) (Tangle relations)

\hspace{0.7cm}(a) $E_{i}E_{i\pm 1}E_{i}=E_{i}.$

\hspace{0.7cm}(b) $T_{i}T_{i\pm 1}E_{i}=E_{i\pm 1}E_{i}$ and $E_{i}T_{i\pm 1}T_{i}=E_{i}E_{i\pm 1}.$

\hspace{0.7cm}(c) For $1\leq j\leq d-1,$ $E_{1}X_{1}^{j}E_{1}=\delta_{j}E_{1}.$

(5) (Kauffman skein relations) $T_{i}-T_{i}^{-1}=(q-q^{-1})(1-E_{i})$ for $1\leq i\leq n-1.$

(6) (Untwisting relations) $T_{i}E_{i}=E_{i}T_{i}=\rho^{-1}E_{i}$ for $1\leq i\leq n-1.$

(7) (Unwrapping relations) $E_{1}X_{1}T_{1}X_{1}=\rho E_{1}=X_{1}T_{1}X_{1}E_{1}.$

(8) (Cyclotomic relation) $(X_1-v_1)(X_1-v_2)\cdots (X_1-v_d)=0.$
\end{definition}

In $\mathscr{B}_{d, n}$, We define inductively the following elements:
\begin{equation}\label{JMur-elements}
X_{i+1} :=T_{i}X_iT_i\quad\mbox{for}~i=1,\ldots,n-1.
\end{equation}
It can be easily checked that the elements $X_1,\ldots,X_n$ commute with each other, and moreover, we have
\begin{equation}\label{JMur-elements1}
E_{i}X_{i}X_{i+1}=X_{i}X_{i+1}E_{i}=E_{i}\quad\mbox{for}~i=1,\ldots,n-1.
\end{equation}

We now define the following elements (see [IMOg2, (2.15)]):
\begin{equation}\label{Baxterized-elements11}
T_{i}(u,v)=T_{i}+\frac{(q-q^{-1})u}{v-u}+\frac{(q-q^{-1})u}{u+\rho qv}E_{i}\quad\mbox{for}~i=1,\ldots,n-1.
\end{equation}
Note that $E_{i}^{2}=\delta_{0}E_{i}$, where $\delta_{0}=\frac{(q^{-1}+\rho^{-1})(\rho q-1)}{q-q^{-1}}.$ By using this, it can be easily checked that (see [IMOg2, (2.17-18)])
\begin{equation}\label{Baxterized-elements111}
T_{i}(u,v)T_{i}(v,u)=f(u,v)\quad\mbox{for}~i=1,\ldots,n-1,
\end{equation}
where
\begin{equation}\label{Baxterized-elements1111}
f(u,v)=f(v,u)=\frac{(u-q^{2}v)(u-q^{-2}v)}{(u-v)^{2}}.
\end{equation}

\subsection{Combinatorics}
$\lambda=(\lambda_{1},\ldots,\lambda_{k})$ is called a partition of $n$ if it is a finite sequence of weakly decreasing nonnegative integers whose sum is $n.$ We set $|\lambda| :=n.$ We shall identify a partition $\lambda$ with a Young diagram, which is the set $$[\lambda] :=\{(i,j)\:|\:i\geq 1~\mathrm{and}~1\leq j\leq \lambda_{i}\}.$$ We shall regard $\lambda$ as a left-justified array of boxes such that there exist $\lambda_{j}$ boxes in the $j$-th row for $j=1,\ldots,k.$ We write $\theta=(a,b)$ if the box $\theta$ lies in row $a$ and column $b.$

Similarly, a $d$-partition of $n$ is an ordered $d$-tuple $\bm{\lambda}=(\lambda^{(1)},\lambda^{(2)},\ldots,\lambda^{(d)})$ of partitions $\lambda^{(k)}$ such that $\sum_{k=1}^{d}|\lambda^{(k)}|=n.$ We denote by $\mathcal{P}_{d}(n)$ the set of $d$-partitions of $n.$ We shall identify a $d$-partition $\bm{\lambda}$ with its Young diagram, which is the ordered $d$-tuple of the Young diagrams of its components. We write $\bm{\theta}=(\theta, s)$ if the box $\theta$ lies in the component $\lambda^{(s)}.$

Assume that $\bm{\lambda}$ and $\bm{\mu}$ are two $d$-partitions. We say that $\bm{\lambda}$ is obtained from $\bm{\mu}$ by adding a box if there exists a pair $(j,t)$ such that $\lambda_{j}^{(t)}=\mu_{j}^{(t)}+1$ and $\lambda_{i}^{(s)}=\mu_{i}^{(s)}$ for $(i,s)\neq (j,t).$ In this case, we will also say that $\bm{\mu}$ is obtained from $\bm{\lambda}$ by removing a box.

Set \[\Lambda_{d,n}^{+} :=\{(l,\bm{\lambda})\:|\:0\leq l\leq \lfloor n/2\rfloor, \bm{\lambda}\in \mathcal{P}_{d}(n-2l)\}.\]

The combinatorial objects appearing in the representation theory of $\mathscr{B}_{d, n}$ will be updown tableaux. For $(f, \bm{\lambda})\in \Lambda_{d,n}^{+},$ an $n$-updown $\bm{\lambda}$-tableau, or more simply an updown $\bm{\lambda}$-tableau, is a sequence $\mathcal{T}=(\mathcal{T}_{1},\mathcal{T}_{2},\ldots,\mathcal{T}_{n})$ of $d$-partitions such that $\mathcal{T}_{n}=\bm{\lambda}$ and $\mathcal{T}_{i}$ is obtained from $\mathcal{T}_{i-1}$ by either adding or removing a box, for $i=1,\ldots,n$, where we set $\mathcal{T}_{0}=\emptyset.$ Let $\mathscr{T}_{n}^{ud}(\bm{\lambda})$ be the set of updown $\bm{\lambda}$-tableaux of $n.$

Suppose that $(f, \bm{\lambda})\in \Lambda_{d,n}^{+}$ and $\mathcal{U}=(\mathcal{U}_{1},\ldots,\mathcal{U}_{n})\in \mathscr{T}_{n}^{ud}(\bm{\lambda}).$ Let
\begin{align}\label{symme-forms}
\mathrm{c}(\mathcal{U}|k)=
\begin{cases}
v_{s}q^{2(j-i)} & \text{if } \mathcal{U}_{k}=\mathcal{U}_{k-1}\cup ((i,j),s),
\\
v_{s}^{-1}q^{2(i-j)} & \text{if } \mathcal{U}_{k-1}=\mathcal{U}_{k}\cup ((i,j),s).
\end{cases}
\end{align}
Given a box $\bm{\alpha}=((i,j),s),$ we define the content of it by
\begin{align}\label{symme-forms11113344}
\mathrm{c}(\mathcal{U}|\bm{\alpha})=
\begin{cases}
v_{s}q^{2(j-i)} & \text{if }\bm{\alpha}\text{ is an addable box of }\mathcal{U},
\\
v_{s}^{-1}q^{2(i-j)} & \text{if }\bm{\alpha}\text{ is a removable box of }\mathcal{U}.
\end{cases}
\end{align}

We give the generalizations of some constructions in [IM, Section 3]. Suppose that $(f, \bm{\lambda})\in \Lambda_{d,n}^{+}$ and $\mathcal{T}=(\mathcal{T}_{1},\ldots,\mathcal{T}_{n})\in \mathscr{T}_{n}^{ud}(\bm{\lambda}).$ Set $\bm{\mu}=\mathcal{T}_{n-1}$ and consider the updown $\bm{\mu}$-tableau $\mathcal{U}=(\mathcal{T}_{1},\ldots,\mathcal{T}_{n-1}).$ We now define two $d$-tuples of infinite matrices
\[M(\mathcal{U})=(m_{1}(\mathcal{U}),\ldots,m_{d}(\mathcal{U}))\quad \mathrm{ and }\quad \overline{M}(\mathcal{U})=(\overline{m}_{1}(\mathcal{U}),\ldots,\overline{m}_{d}(\mathcal{U})),\]
here the rows and columns of each $m_{s}(\mathcal{U})$ or $\overline{m}_{s}(\mathcal{U})$ are labelled by positive integers and only a finite number of entries in each of the matrices are nonzero. The entry $m_{ij}^{s}$ of the matrix $m_{s}(\mathcal{U})$ (respectively, the entry $\overline{m}_{ij}^{s}$ of the matrix $\overline{m}_{s}(\mathcal{U})$) equals the number of times that the box $((i,j),s)$ is added (respectively, removed) in the sequence $(\emptyset, \mathcal{T}_{1},\ldots, \mathcal{T}_{n-1}).$

For each $k\in \mathbb{Z}$ and $1\leq s\leq d,$ we define two nonnegative integers $d_{k}^{s}=d_{k}(m_{s}(\mathcal{U}))$ and $\overline{d}_{k}^{s}=d_{k}(\overline{m}_{s}(\mathcal{U}))$ as the sums of the entries of the matrices $m_{s}(\mathcal{U})$ and $\overline{m}_{s}(\mathcal{U})$ on the $k$-th diagonal, that is,
\begin{equation}\label{dsk-dskbar}
d_{k}^{s}=\sum_{j-i=k}m_{ij}^{s}\quad\mathrm{ and }\quad \overline{d}_{k}^{s}=\sum_{j-i=k}\overline{m}_{ij}^{s}.
\end{equation}

Furthermore, we define the indexes $g_{k}^{s}=g_{k}(m_{s}(\mathcal{U}))$ and $\overline{g}_{k}^{s}=g_{k}(\overline{m}_{s}(\mathcal{U}))$ as follows:
\begin{equation}\label{index-indexbar}
g_{k}^{s}=\delta_{k0}+d_{k-1}^{s}+d_{k+1}^{s}-2d_{k}^{s}\quad\mathrm{ and }\quad \overline{g}_{k}^{s}=\overline{d}_{k-1}^{s}+\overline{d}_{k+1}^{s}-2\overline{d}_{k}^{s}.
\end{equation}

Finally, we define some integer $p_{1},\ldots,p_{n}$ associated to $\mathcal{T}$ inductively such that $p_{k}$ depends only on the first $k$ $d$-partitions $(\mathcal{T}_{1},\ldots,\mathcal{T}_{k})$ of $\mathcal{T}.$ Therefore, it is enough to define $p_{n}.$ We set
\begin{equation}\label{integer-indexbar}
p_{n}=1-g_{k_{n}}(m_{s_{n}}(\mathcal{U}))
\end{equation}
if $\mathcal{T}_{n}$ is obtained from $\mathcal{T}_{n-1}$ by adding a box $((i_n,j_n),s_n)$, where $k_n=j_n-i_n;$
\begin{equation}\label{integer-indexbar11}
p_{n}=1-g_{k_{n}'}(\overline{m}_{s_{n}'}(\mathcal{U}))
\end{equation}
if $\mathcal{T}_{n}$ is obtained from $\mathcal{T}_{n-1}$ by removing a box $((i_{n}',j_{n}'),s_{n}')$, where $k_{n}'=j_{n}'-i_{n}'.$

Assume that $(f, \bm{\lambda})\in \Lambda_{d,n}^{+},$ $\mathcal{T}=(\mathcal{T}_{1},\ldots,\mathcal{T}_{n})$ is an $n$-updown $\bm{\lambda}$-tableau and that
$\mathcal{U}=(\mathcal{T}_{1},\ldots,\mathcal{T}_{n-1}).$ We then define the element $f(\mathcal{T})$ inductively by
\begin{equation}\label{hooklength-indexbar11}
f(\mathcal{T})=f(\mathcal{U})\varphi(\mathcal{U}, \mathcal{T}),
\end{equation}
where
\begin{equation*}
\varphi(\mathcal{U}, \mathcal{T})=\prod_{\substack{k\neq k_{n}\\k\in \mathbb{Z}}}(q^{2k_{n}}-q^{2k})^{g_{k}^{s_{n}}}\prod_{\substack{1\leq t\leq d; t\neq s_{n}\\k\in \mathbb{Z}}}\hspace{-2mm}(v_{s_{n}}q^{2k_{n}}-v_{t}q^{2k})^{g_{k}^{t}}  \prod_{\substack{1\leq r\leq d\\k\in \mathbb{Z}}}(v_{s_{n}}q^{2k_{n}}-v_{r}^{-1}q^{-2k})^{\overline{g}_{k}^{r}}
\end{equation*}
if $\mathcal{T}_{n}$ is obtained from $\mathcal{T}_{n-1}$ by adding a box $((i_n,j_n),s_n)$, where $k_n=j_n-i_n;$
\begin{equation*}
\varphi(\mathcal{U}, \mathcal{T})=\prod_{\substack{k\neq k_{n}'\\k\in \mathbb{Z}}}(q^{-2k_{n}'}-q^{-2k})^{\overline{g}_{k}^{s_{n}'}}\hspace{-1.5mm}
\prod_{\substack{1\leq t\leq d; t\neq s_{n}'\\k\in \mathbb{Z}}}\hspace{-2mm}(v_{s_{n}'}^{-1}q^{-2k_{n}'}-v_{t}^{-1}q^{-2k})^{\overline{g}_{k}^{t}}  \prod_{\substack{1\leq r\leq d\\k\in \mathbb{Z}}}(v_{s_{n}'}^{-1}q^{-2k_{n}'}-v_{r}q^{2k})^{g_{k}^{r}}
\end{equation*}
if $\mathcal{T}_{n}$ is obtained from $\mathcal{T}_{n-1}$ by removing a box $((i_{n}',j_{n}'),s_{n}')$, where $k_{n}'=j_{n}'-i_{n}'.$

In the special situation when $f=0,$ that is, $\bm{\lambda}$ is a $d$-partition of $n,$ there is a natural bijection between the set of $n$-updown $\bm{\lambda}$-tableaux and the set of standard $\bm{\lambda}$-tableaux defined in [DJM, Definition (3.10)]. The following proposition is inspired by [IM, Proposition 3.3] and can be proved similarly.
\begin{proposition}\label{special-propo}
If $\bm{\lambda}$ is a $d$-partition of $n$ and $\mathcal{T}=(\mathcal{T}_{1},\ldots,\mathcal{T}_{n})$ is an $n$-updown $\bm{\lambda}$-tableau, then $p_1,\ldots,p_{n}$ are all equal to zero, and $f(\mathcal{T})$ is exactly equal to $\emph{F}_{\bm{\lambda}}^{-1}$ defined in $[\emph{OgPA}2, \emph{Section } 2.2(12)]$ when $d=m.$
\end{proposition}

\subsection{Idempotents of $\mathscr{B}_{d, n}$}
Following [RuXu, Definition 3.4], we say that $\mathscr{B}_{d, n}$ is generic if the parameters $v_i$, $1\leq i\leq d$, and $q$ satisfy the conditions (1) the order $o(q^{2})$ of $q^{2}$ satisfies $o(q^{2})> 2n$; (2) $|r|\geq 2n$ whenever there exists $r\in \mathbb{Z}$ such that either $v_{i}v_{j}^{\pm 1}=q^{2r}$ for $i\neq j,$ or $v_{i}=\pm q^{r}.$ Following [WiYu1, Corollary 4.5], we say that $\mathscr{B}_{d, n}$ is admissible if the set $\{E_{1}, E_{1}X_{1},\ldots,E_{1}X_{1}^{d-1}\}$ is linearly independent in $\mathscr{B}_{d, 2}.$ It has been proved by Goodman [Go2, Theorem 4.4] that this admissible condition coincides with the $\bm{\mathrm{u}}$-admissible condition defined in [RuXu, Definition 2.27].

From now on, we always assume that $\mathscr{B}_{d, n}$ is generic and admissible. Thus, by [RuXu, Lemma 3.5], we have $\mathcal{S}=\mathcal{T}$ if and only if $\mathrm{c}(\mathcal{S}|k)=\mathrm{c}(\mathcal{T}|k)$ for all $1\leq k\leq n.$ Therefore, the set $\{X_1,\ldots,X_n\},$ as a family of JM-elements for $\mathscr{B}_{d, n}$ in the abstract sense defined in [Ma, Definition 2.4], satisfies the separation condition associated to the weakly cellular basis of $\mathscr{B}_{d, n}$ constructed in [RuXu, Theorem 4.19]. Note that the results in [Ma] also hold for $\mathscr{B}_{d, n}$ with respect to the weakly cellular basis. In particular, we can construct the primitive idempotents of $\mathscr{B}_{d, n}$ following the arguments in [Ma, Section 3].

For each $1\leq k\leq n,$ we define the following set:
\[\mathcal{R}(k) :=\{\mathrm{c}(\mathcal{S}|k)\:|\:\mathcal{S}\in \mathscr{T}_{n}^{ud}(\bm{\lambda})
\text{ for some }(f, \bm{\lambda})\in \Lambda_{d,n}^{+}\}.\]
Suppose that $(f, \bm{\lambda})\in \Lambda_{d,n}^{+}$ and $\mathcal{T}\in \mathscr{T}_{n}^{ud}(\bm{\lambda}).$ We set
\begin{equation}\label{hooklength-idempotentelement11}
E_{\mathcal{T}}=\prod_{k=1}^{n}\bigg(\prod_{\substack{c\in \mathcal{R}(k)\\c\neq \mathrm{c}(\mathcal{T}|k)}}\frac{X_{k}-c}{\mathrm{c}(\mathcal{T}|k)-c}
\bigg).
\end{equation}
By standard arguments in [Ma, Section 3], the elements $\{E_{\mathcal{T}}\:|\:\mathcal{T}\in \mathscr{T}_{n}^{ud}(\bm{\lambda})
\text{ for some }(f, \bm{\lambda})\in \Lambda_{d,n}^{+}\}$ form a complete set of pairwise orthogonal primitive idempotents of $\mathscr{B}_{d, n}.$ Moreover, the elements $X_1,\ldots,X_n$ generate a maximal commutative subalgebra of $\mathscr{B}_{d, n}.$ We also have
\begin{equation}\label{hooklength-idempotentelement1111}
X_{k}E_{\mathcal{T}}=E_{\mathcal{T}}X_{k}=\mathrm{c}(\mathcal{T}|k)E_{\mathcal{T}}.
\end{equation}

\section{Fusion procedure for cyclotomic BMW algebras}
Assume that $(f, \bm{\lambda})\in \Lambda_{d,n}^{+}$ and that $\mathcal{T}=(\mathcal{T}_{1},\ldots,\mathcal{T}_{n})$ is an $n$-updown $\bm{\lambda}$-tableau. Set $\bm{\mu}=\mathcal{T}_{n-1}$ and $\mathcal{U}=(\mathcal{T}_{1},\ldots,\mathcal{T}_{n-1})$ as an updown $\bm{\mu}$-tableau. Let $\bm{\theta}$ be the box that is addable to or removable from $\bm{\mu}$ to get $\bm{\lambda}.$ For simplicity, we set $\mathrm{c}_{k} :=\mathrm{c}(\mathcal{T}|k).$ By \eqref{hooklength-idempotentelement11}, we can rewrite $E_{\mathcal{T}}$ inductively as follows:
\begin{equation}\label{idempotentele-induc}
E_{\mathcal{T}}=E_{\mathcal{U}}\frac{(X_{n}-b_1)\cdots (X_{n}-b_k)}{(\mathrm{c}_{n}-b_1)\cdots (\mathrm{c}_{n}-b_k)},
\end{equation}
where $b_1,\ldots,b_k$ are the contents of all boxes except $\bm{\theta},$ which can be addable to or removable from $\bm{\mu}$ to get a $d$-partition.

We denote by $\{\Lambda_{1},\ldots,\Lambda_{h}\}$ the set of all $d$-partitions obtained from $\bm{\mu}$ by adding a box or removing one. Set $\mathscr{T}_{j} :=(\mathcal{T}_{1},\ldots,\mathcal{T}_{n-1},\Lambda_{j})$ for $1\leq j\leq h.$ Note that $\mathcal{T}\in \{\mathscr{T}_{1},\ldots,\mathscr{T}_{h}\}.$ Since $\mathscr{B}_{d, n}$ is generic, hence it is semisimple. By [RuSi3, (4.16)] we have
\begin{equation}\label{sum-formula11}
E_{\mathcal{U}}=\sum_{j=1}^{h}E_{\mathscr{T}_{j}}.
\end{equation}
The property \eqref{hooklength-idempotentelement1111} implies that the following rational function
\begin{equation}\label{rational-function11}
E_{\mathcal{U}}\frac{u-\text{c}_n}{u-X_{n}}
\end{equation}
is regular at $u=\text{c}_n,$ and by \eqref{sum-formula11}, we get
\begin{equation}\label{sum-function1111}
E_{\mathcal{U}}\frac{u-\text{c}_n}{u-X_{n}}\Big|_{u=\text{c}_n}=E_{\mathcal{T}}.
\end{equation}

For $1\leq i\leq n-1,$ we set
\begin{align}\label{Q-function41}
Q_{i}(u,v;c) :=T_{i}+\frac{q-q^{-1}}{\rho^{-1}cuv-1}+\frac{q-q^{-1}}{1+qcuv}E_{i}.
\end{align}
Let $\phi_{1}(u) :=\frac{cuX_{1}-\rho}{u-X_1}.$ For $k=2,\ldots,n$, we set
\begin{align}\label{phi-function42}
\phi_k(u_1,\ldots,u_{k-1},u)& :=Q_{k-1}(u_{k-1},u;c)\phi_{k-1}(u_1,\ldots,u_{k-2},u)T_{k-1}(u_{k-1}, u)\notag\\
=Q_{k-1}&(u_{k-1},u;c)\cdots Q_{1}(u_{1},u;c)\phi_{1}(u)T_{1}(u_{1}, u)\cdots T_{k-1}(u_{k-1}, u).
\end{align}

From now on, we always set $c :=-q^{-1}.$ The following lemma is inspired by [IMOg2, Lemma 1] and can be proved similarly.
\begin{lemma}\label{phi-phi-phi111}
Assume that $n\geq 1.$ We have
\begin{align}\label{F-PhiEu43}
E_{\mathcal{U}}\phi_n(\mathrm{c}_1,\ldots,\mathrm{c}_{n-1},u)\prod_{r=1}^{n-1}f(u, \mathrm{c}_{r})^{-1}=E_{\mathcal{U}}\frac{cuX_{n}-\rho}{u-X_n}.
\end{align}
\end{lemma}
\begin{proof}
We shall prove \eqref{F-PhiEu43} by induction on $n.$ For $n=1,$ the situation is trivial.

We set
\begin{align}\label{phi-function421}
\phi'_n(\mathrm{c}_1,&\ldots,\mathrm{c}_{n-1},u)\notag\\
&=Q_{n-1}(\mathrm{c}_{n-1},u;c)\cdots Q_{1}(\mathrm{c}_{1},u;c)\phi_{1}(u)T_{1}(u, \mathrm{c}_{1})^{-1}\cdots T_{n-1}(u, \mathrm{c}_{n-1})^{-1}.
\end{align}
By \eqref{Baxterized-elements111} and \eqref{phi-function421}, in order to show \eqref{F-PhiEu43}, it suffices to prove that
\begin{align}\label{F-PhiEu4321}
E_{\mathcal{U}}\phi'_n(\mathrm{c}_1,\ldots,\mathrm{c}_{n-1},u)=E_{\mathcal{U}}\frac{cuX_{n}-\rho}{u-X_n}.
\end{align}
By the induction hypothesis, it boils down to proving the following equality:
\begin{align}\label{EUEU-PhiEu5}
E_{\mathcal{U}}Q_{n-1}(\mathrm{c}_{n-1},u;c)\frac{cuX_{n-1}-\rho}{u-X_{n-1}}T_{n-1}(u, \mathrm{c}_{n-1})^{-1}=E_{\mathcal{U}}\frac{cuX_{n}-\rho}{u-X_n}.
\end{align}
Since $X_{n}$ commutes with $E_{\mathcal{U}},$ we can rewrite \eqref{EUEU-PhiEu5} as follows:
\begin{align}\label{EUEU-PhiEu6}
E_{\mathcal{U}}(u-X_n)Q_{n-1}(&\mathrm{c}_{n-1},u;c)(cuX_{n-1}-\rho)\notag\\
&=E_{\mathcal{U}}(cuX_{n}-\rho)T_{n-1}(u, \mathrm{c}_{n-1})(u-X_{n-1}).
\end{align}
By \eqref{Baxterized-elements11} and \eqref{Q-function41}, the equality \eqref{EUEU-PhiEu6} becomes
\begin{align}\label{EUEU-PhiEu7}
E_{\mathcal{U}}(u&-X_n)\Big(T_{n-1}+\frac{q-q^{-1}}{\rho^{-1}cu\mathrm{c}_{n-1}-1}+\frac{q-q^{-1}}{1+qcu\mathrm{c}_{n-1}}E_{n-1}\Big)(cuX_{n-1}-\rho)\notag\\
&=E_{\mathcal{U}}(cuX_{n}-\rho)\Big(T_{n-1}+\frac{(q-q^{-1})u}{\mathrm{c}_{n-1}-u}+\frac{(q-q^{-1})u}{u+\rho q\mathrm{c}_{n-1}}E_{n-1}\Big)(u-X_{n-1}).
\end{align}

By definition, we have $T_{n-1}X_{n-1}=X_{n}T_{n-1}-(q-q^{-1})X_{n}+(q-q^{-1})X_{n}E_{n-1}.$ Thus, we get that \eqref{EUEU-PhiEu7} is equivalent to
\begin{align}\label{EUEU-PhiEu8}
E_{\mathcal{U}}(u&-X_n)\Big(cu(X_{n}T_{n-1}-(q-q^{-1})X_{n}+(q-q^{-1})X_{n}E_{n-1})-\rho T_{n-1}\notag\\
&\hspace{2cm}+(q-q^{-1})\rho+\frac{q-q^{-1}}{1+qcu\mathrm{c}_{n-1}}E_{n-1}(cuX_{n-1}-\rho)\Big)\notag\\
&=E_{\mathcal{U}}(cuX_{n}-\rho)\Big(-X_{n}T_{n-1}+(q-q^{-1})X_{n}-(q-q^{-1})X_{n}E_{n-1}+uT_{n-1}\notag\\
&\hspace{2cm}-(q-q^{-1})u+\frac{(q-q^{-1})u}{u+\rho q\mathrm{c}_{n-1}}E_{n-1}(u-X_{n-1})\Big).
\end{align}
Since we have
\begin{align*}
cu^{2}&X_{n}T_{n-1}-cuX_{n}^{2}T_{n-1}-(q-q^{-1})cuX_{n}(u-X_n)-\rho(u-X_n)(T_{n-1}-(q-q^{-1}))\notag\\
=&-cuX_{n}^{2}T_{n-1}+\rho X_nT_{n-1}+(q-q^{-1})(cuX_{n}-\rho)X_{n}+u(cuX_{n}-\rho)(T_{n-1}-(q-q^{-1})),
\end{align*}
it is easy to see that the equality \eqref{EUEU-PhiEu8} comes down to the following equality:
\begin{align}\label{EUEU-PhiEu9}
E_{\mathcal{U}}&(u-X_n)\Big(cuX_{n}E_{n-1}+\frac{1}{1+qcu\mathrm{c}_{n-1}}E_{n-1}(cuX_{n-1}-\rho)\Big)\notag\\
&=E_{\mathcal{U}}(cuX_{n}-\rho)\Big(-X_{n}E_{n-1}+\frac{u}{u+\rho q\mathrm{c}_{n-1}}E_{n-1}(u-X_{n-1})\Big).
\end{align}

By definition, we have $E_{\mathcal{U}}X_{n-1}=\mathrm{c}_{n-1}E_{\mathcal{U}}.$ Hence, we get $E_{\mathcal{U}}X_{n}E_{n-1}=\frac{1}{\mathrm{c}_{n-1}}E_{\mathcal{U}}E_{n-1}$ by \eqref{JMur-elements1}.
According to this, by comparing the coefficients of the terms involving $E_{\mathcal{U}}E_{n-1}X_{n-1}$, we see that it suffices to show that
\begin{align}\label{EUEU-PhiEu11}
\frac{1}{1+qcu\mathrm{c}_{n-1}}\cdot \frac{cu^{2}\mathrm{c}_{n-1}-cu}{\mathrm{c}_{n-1}}=\frac{u}{u+\rho q\mathrm{c}_{n-1}}\cdot \frac{\rho\mathrm{c}_{n-1}-cu}{\mathrm{c}_{n-1}}.
\end{align}
By comparing the coefficients of the terms involving $E_{\mathcal{U}}E_{n-1}$, it suffices to show that
\begin{align}\label{EUEU-PhiEu12}
\frac{cu^{2}-\rho}{\mathrm{c}_{n-1}}+\frac{1}{1+qcu\mathrm{c}_{n-1}}\cdot \frac{\rho-\rho u\mathrm{c}_{n-1}}{\mathrm{c}_{n-1}}=
\frac{u}{u+\rho q\mathrm{c}_{n-1}}\cdot \frac{cu^{2}-\rho u\mathrm{c}_{n-1}}{\mathrm{c}_{n-1}}.
\end{align}
Noting that $c=-q^{-1},$ it is easy to verify that \eqref{EUEU-PhiEu11} and \eqref{EUEU-PhiEu12} are true. Thus, \eqref{EUEU-PhiEu9} holds. The lemma is proved.
\end{proof}

Let $\overline{\phi}_{1}(u) :=(u-v_{1})\cdots (u-v_{d})\frac{cuX_{1}-\rho}{u-X_1}.$ For $k=2,\ldots,n$, we set
\begin{align}\label{phi-function424242}
\overline{\phi}_k(u_1,\ldots,u_{k-1},u)& :=Q_{k-1}(u_{k-1},u;c)\overline{\phi}_{k-1}(u_1,\ldots,u_{k-2},u)T_{k-1}(u_{k-1}, u)\notag\\
=Q_{k-1}&(u_{k-1},u;c)\cdots Q_{1}(u_{1},u;c)\overline{\phi}_{1}(u)T_{1}(u_{1}, u)\cdots T_{k-1}(u_{k-1}, u).
\end{align}

We also define the following rational function:
\begin{align}\label{Phi-function111}
\Phi(u_1,\ldots,u_n) :=\overline{\phi}_1(u_1)\cdots \overline{\phi}_{n-1}(u_1,\ldots,u_{n-1})\overline{\phi}_n(u_1,\ldots,u_{n}).
\end{align}
Recall that the integers $p_{1},\ldots,p_{n}$ associated to $\mathcal{T}$ have been defined as in \eqref{integer-indexbar} or \eqref{integer-indexbar11}.

Now we can state the main result of this paper.

\begin{theorem}\label{main-theorem11112}
The idempotent $E_{\mathcal{T}}$ of $\mathscr{B}_{d, n}$ corresponding to an $n$-updown $\bm{\lambda}$-tableau $\mathcal{T}$ can be derived by the following consecutive evaluations$:$
\begin{equation}\label{idempotents111}
E_{\mathcal{T}}=\frac{1}{f(\mathcal{T})}\Big(\prod_{k=1}^{n}\frac{(u_{k}-\mathrm{c}_{k})^{p_{k}}}{cu_{k}\mathrm{c}_{k}-\rho}\Big)
\Phi(u_1,\ldots,u_n)\Big|_{u_{1}=\emph{c}_1}\cdots\Big|_{u_{n}=\emph{c}_{n}}.
\end{equation}
\end{theorem}
\begin{proof}
We shall prove the theorem by induction on $n.$ For $n=1,$ we have $p_{1}=0$ by Proposition \ref{special-propo}. Thus, we get that the right-hand side of \eqref{idempotents111} is equal to
\begin{align}\label{n-1-istrue}
\frac{1}{f(\mathcal{T})}&\frac{(u_{1}-v_{1})\cdots (u_{1}-v_{d})}{cu_{1}\mathrm{c}_{1}-\rho}
\frac{cu_{1}X_{1}-\rho}{u_{1}-X_1}\Big|_{u_{1}=\mathrm{c}_1}\notag\\
&=\frac{1}{f(\mathcal{T})}\frac{(u_{1}-v_{1})\cdots (u_{1}-v_{d})}{u_{1}-\mathrm{c}_{1}}\frac{u_{1}-\mathrm{c}_{1}}{cu_{1}\mathrm{c}_{1}-\rho}\frac{cu_{1}X_{1}-\rho}{u_{1}-X_1}\Big|_{u_{1}=\mathrm{c}_1}.
\end{align}
Moreover, by \eqref{hooklength-indexbar11}, we have \[f(\mathcal{T})=\prod_{1\leq k\leq d;v_{k}\neq \mathrm{c}_{1}}(\mathrm{c}_{1}-v_{k}).\]
Therefore, it is easy to see that \eqref{n-1-istrue} is equal to $E_{\mathcal{T}}$ by \eqref{hooklength-idempotentelement1111} and \eqref{sum-function1111}.

For $n\geq 2,$ by the induction hypothesis we can write the right-hand side of \eqref{idempotents111} as follows:
\begin{align}\label{n-1-istrue2}
\frac{f(\mathcal{U})}{f(\mathcal{T})}\frac{(u_{n}-\mathrm{c}_{n})^{p_{n}}}{cu_{n}\mathrm{c}_{n}-\rho}E_{\mathcal{U}}
\overline{\phi}_n(\mathrm{c}_{1},\ldots,\mathrm{c}_{n-1},u_n)\Big|_{u_{n}=\mathrm{c}_{n}}.
\end{align}
Note that $\overline{\phi}_n(\mathrm{c}_{1},\ldots,\mathrm{c}_{n-1},u_n)=(u_{n}-v_{1})\cdots (u_{n}-v_{d})\phi_n(\mathrm{c}_{1},\ldots,\mathrm{c}_{n-1},u_n).$ By \eqref{F-PhiEu43}, we can rewrite the expression \eqref{n-1-istrue2} as
\begin{align}\label{n-1-istrue3}
\frac{f(\mathcal{U})}{f(\mathcal{T})}\frac{(u_{n}-\mathrm{c}_{n})^{p_{n}}}{cu_{n}\mathrm{c}_{n}-\rho}(u_{n}-v_{1})\cdots (u_{n}-v_{d})\prod_{r=1}^{n-1}f(u_{n}, \mathrm{c}_{r})E_{\mathcal{U}}\frac{cu_{n}X_{n}-\rho}{u_{n}-X_n}\Big|_{u_{n}=\mathrm{c}_{n}}.
\end{align}

By \eqref{hooklength-indexbar11}, we see that
\begin{align*}
\frac{f(\mathcal{U})}{f(\mathcal{T})}(u_{n}&-v_{1})\cdots (u_{n}-v_{d})\prod_{r=1}^{n-1}f(u_{n}, \mathrm{c}_{r})(u_{n}-\mathrm{c}_{n})^{p_{n}-1}\notag\\
&=\frac{f(\mathcal{U})}{f(\mathcal{T})}(u_{n}-v_{1})\cdots (u_{n}-v_{d})\prod_{r=1}^{n-1}\frac{(u_{n}-q^{2}\mathrm{c}_{r})(u_{n}-q^{-2}\mathrm{c}_{r})}{(u_{n}-\mathrm{c}_{r})^{2}}(u_{n}-\mathrm{c}_{n})^{p_{n}-1}
\end{align*}
is regular at $u_n=\mathrm{c}_{n}$ and is equal to $1.$ Thus, the expression \eqref{n-1-istrue3} equals
\begin{align}\label{n-1-istrue4}
E_{\mathcal{U}}\frac{u_{n}-\mathrm{c}_{n}}{u_{n}-X_n}\frac{cu_{n}X_{n}-\rho}{cu_{n}\mathrm{c}_{n}-\rho}\Big|_{u_{n}=\mathrm{c}_{n}}.
\end{align}
By \eqref{sum-function1111}, we see that \eqref{n-1-istrue4} is equal to
\begin{align}\label{n-1-istrue5}
E_{\mathcal{T}}\frac{cu_{n}X_{n}-\rho}{cu_{n}\mathrm{c}_{n}-\rho}\Big|_{u_{n}=\mathrm{c}_{n}}.
\end{align}
By \eqref{hooklength-idempotentelement1111}, we have $E_{\mathcal{T}}X_{n}=\mathrm{c}_{n}E_{\mathcal{T}}.$ Thus, we get that the expression \eqref{n-1-istrue5}, that is, the right-hand side of \eqref{idempotents111} equals $E_{\mathcal{T}}.$
 \end{proof}

\begin{remark}\label{remark111}
Let $\mathscr{H}_{d, n}$ be the cyclotomic Hecke algebra defined in [AK]. It has been proved in [RuXu, Proposition 4.1] that $\mathscr{H}_{d, n}$ is isomorphic to the quotient of $\mathscr{B}_{d, n}$ by the two-sided ideal generated by all $E_{i}.$ In the process of taking quotient, the parameter $\rho$ disappears; however, the parameter $c$ is reserved and can be arbitrary. If we replace the $T_{i}(u, v),$ $Q_{i}(u,v;c),$ $\phi_{1}(u)$ in \eqref{phi-function42} with
\begin{align*}
\overline{T}_{i}(u,v)=T_{i}+\frac{(q-q^{-1})u}{v-u},\quad \overline{Q}_{i}(u,v;c) :=T_{i}+\frac{q-q^{-1}}{cuv-1},\quad \psi_{1}(u) :=\frac{cuX_{1}-1}{u-X_1},
\end{align*}
it is easy to see that the analogue of Lemma \ref{phi-phi-phi111} holds.

Let $\overline{\psi}_{1}(u) :=(u-v_{1})\cdots (u-v_{d})\frac{cuX_{1}-1}{u-X_1},$ and for $k=2,\ldots,n$, set
\begin{align*}
\overline{\psi}_k(u_1,\ldots,u_{k-1},u)& :=\overline{Q}_{k-1}(u_{k-1},u;c)\overline{\psi}_{k-1}(u_1,\ldots,u_{k-2},u)\overline{T}_{k-1}(u_{k-1}, u).
\end{align*}

We also define a rational function by
\begin{align*}
\Upsilon(u_1,\ldots,u_n) :=\overline{\psi}_1(u_1)\cdots \overline{\psi}_{n-1}(u_1,\ldots,u_{n-1})\overline{\psi}_n(u_1,\ldots,u_{n}).
\end{align*}
Then it is easy to see that the analogue of Theorem \ref{main-theorem11112} is true. Thus, we get a one-parameter family of the fusion procedures for cyclotomic Hecke algebras, generalizing the results obtained in [OgPA2].
\end{remark}

\section{Appendix. Fusion procedure for cyclotomic Nazarov-Wenzl algebras}
When studying the representations of Brauer algebras, Nazarov [Na1] introduced a class of infinite dimensional algebras under the name affine Wenzl algebras. In order to study finite dimensional irreducible representations of affine Wenzl algebras, Ariki, Mathas and Rui [AMR] defined the finite dimensional quotients of them, known as the cyclotomic Nazarov-Wenzl algebras. Cyclotomic Nazarov-Wenzl algebras are related to degenerate cyclotomic Hecke algebras just in the same way that cyclotomic BMW algebras are connected with cyclotomic Hecke algebras. Cyclotomic Nazarov-Wenzl algebras have been studied by many authors; see [Go3-4, RuSi1-2, Xu] and so on.

\subsection{Cyclotomic Nazarov-Wenzl algebras}
\begin{definition}
Suppose that $\mathbb{K}$ is an algebraically closed field containing $\omega_{j}$ ($0\leq j\leq d-1$), $v_i$ ($1\leq i\leq d$), and the invertible element $2.$\vskip2mm

Fix $n\geq 1.$ The cyclotomic Nazarov-Wenzl algebra $\mathscr{W}_{d, n}$ is the $\mathbb{K}$-algebra generated by the elements $S_{i}, E_{i}$ ($1\leq i\leq n-1$) and $X_{j}$ ($1\leq j\leq n$) satisfying the following relations:\vskip2mm

(1) (Involutions) $S_{i}^{2}=1$ for $1\leq i\leq n-1.$

(2) (Idempotent relations) $E_{i}^{2}=\omega_{0} E_{i}$ for $1\leq i\leq n-1.$

(3) (Affine braid relations)

\hspace{0.7cm}(a) $S_{i}S_{i+1}S_{i}=S_{i+1}S_{i}S_{i+1}$ and $S_{i}S_{j}=S_{j}S_{i}$ if $|i-j|\geq 2.$

\hspace{0.7cm}(b) $S_{i}X_{j}=X_{j}S_{i}$ if $j\neq i, i+1.$

(4) (Tangle relations)

\hspace{0.7cm}(a) $E_{i}E_{i\pm 1}E_{i}=E_{i}.$

\hspace{0.7cm}(b) $S_{i}S_{i\pm 1}E_{i}=E_{i\pm 1}E_{i}$ and $E_{i}S_{i\pm 1}S_{i}=E_{i}E_{i\pm 1}.$

\hspace{0.7cm}(c) For $1\leq k\leq d-1,$ $E_{1}X_{1}^{k}E_{1}=\omega_{k}E_{1}.$

(5) (Untwisting relations) $S_{i}E_{i}=E_{i}S_{i}=E_{i}$ for $1\leq i\leq n-1.$

(6) (Skein relations) $S_{i}X_{i}-X_{i+1}S_{i}=E_{i}-1$ for $1\leq i\leq n-1.$

(7) (Anti-symmetry relations) $E_{i}(X_{i}+X_{i+1})=(X_{i}+X_{i+1})E_{i}=0$ for $1\leq i\leq n-1.$

(8) (Commutative relations)

\hspace{0.7cm}(a) $S_{i}E_{j}=E_{j}S_{i}$ and $E_{i}E_{j}=E_{j}E_{i}$ if $|i-j|\geq 2.$

\hspace{0.7cm}(b) $E_{i}X_{j}=X_{j}E_{i}$ if $j\neq i, i+1.$

\hspace{0.7cm}(c) $X_{i}X_{j}=X_{j}X_{i}$ for $1\leq i,j \leq n.$

(9) (Cyclotomic relation) $(X_1-v_1)(X_1-v_2)\cdots (X_1-v_d)=0.$
\end{definition}

We define the following elements:
\begin{equation}\label{Baxterized-elements11cde}
S_{i}(u,v)=S_{i}+\frac{1}{v-u}-\frac{1}{v-u+\frac{\omega_{0}}{2}-1}E_{i}\quad\mbox{for}~1\leq i\leq n-1.
\end{equation}
By using the fact that $E_{i}^{2}=\omega_{0}E_{i}$, we can easily get
\begin{equation}\label{Baxterized-elements111cde}
S_{i}(u,v)S_{i}(v,u)=g(u,v)\quad\mbox{for}~1\leq i\leq n-1,
\end{equation}
where
\begin{equation}\label{Baxterized-elements1111cde}
g(u,v)=g(v,u)=\frac{(u-v+1)(u-v-1)}{(u-v)^{2}}.
\end{equation}

\subsection{Combinatorics}
Suppose that $(f, \bm{\lambda})\in \Lambda_{d,n}^{+}$ and $\mathfrak{s}=(\mathfrak{s}_{1},\ldots,\mathfrak{s}_{n})\in \mathscr{T}_{n}^{ud}(\bm{\lambda}).$ We can define the integers $d_{k}^{s},$ $\overline{d}_{k}^{s},$ $g_{k}^{s},$ $\overline{g}_{k}^{s}$ and some integers $p_{1},\ldots,p_{n}$ associated to $\mathfrak{s}$ in exactly the same way as those related to some $\mathcal{T}$ defined in Subsection 2.2. We shall follow the notations and only emphasize the differences.

Set
\begin{align}\label{symme-formscde}
\mathrm{c}(\mathfrak{s}|k)=
\begin{cases}
v_{s}+j-i & \text{if } \mathfrak{s}_{k}=\mathfrak{s}_{k-1}\cup ((i,j),s),
\\
-v_{s}+i-j & \text{if } \mathfrak{s}_{k-1}=\mathfrak{s}_{k}\cup ((i,j),s).
\end{cases}
\end{align}
Given a box $\bm{\beta}=((i,j),s),$ we define the content of it by
\begin{align}\label{symme-forms11113344cde}
\mathrm{c}(\mathcal{U}|\bm{\beta})=
\begin{cases}
v_{s}+j-i & \text{if }\bm{\beta}\text{ is an addable box of }\mathfrak{s},
\\
-v_{s}+i-j & \text{if }\bm{\beta}\text{ is a removable box of }\mathfrak{s}.
\end{cases}
\end{align}

Assume that $(f, \bm{\lambda})\in \Lambda_{d,n}^{+},$ $\mathfrak{t}=(\mathfrak{t}_{1},\ldots,\mathfrak{t}_{n})$ is an $n$-updown $\bm{\lambda}$-tableau and that
$\mathfrak{u}=(\mathfrak{t}_{1},\ldots,\mathfrak{t}_{n-1}).$ We then define the element $g(\mathfrak{t})$ inductively by
\begin{equation}\label{hooklength-indexbar11cde}
g(\mathfrak{t})=g(\mathfrak{u})\psi(\mathfrak{u}, \mathfrak{t}),
\end{equation}
where
\begin{equation*}
\psi(\mathfrak{u}, \mathfrak{t})=\prod_{\substack{k\neq k_{n}\\k\in \mathbb{Z}}}(k_{n}-k)^{g_{k}^{s_{n}}}\prod_{\substack{1\leq t\leq d; t\neq s_{n}\\k\in \mathbb{Z}}}\hspace{-2mm}(v_{s_{n}}-v_{t}+k_{n}-k)^{g_{k}^{t}}\prod_{\substack{1\leq r\leq d\\k\in \mathbb{Z}}}(v_{s_{n}}+v_{r}+k_{n}+k)^{\overline{g}_{k}^{r}}
\end{equation*}
if $\mathfrak{t}_{n}$ is obtained from $ \mathfrak{t}_{n-1}$ by adding a box $((i_n,j_n),s_n)$, where $k_n=j_n-i_n;$
\begin{equation*}
\psi(\mathfrak{u}, \mathfrak{t})=\prod_{\substack{k\neq k_{n}'\\k\in \mathbb{Z}}}(-k_{n}'+k)^{\overline{g}_{k}^{s_{n}'}}\prod_{\substack{1\leq t\leq d; t\neq s_{n}'\\k\in \mathbb{Z}}}(-v_{s_{n}'}+v_{t}-k_{n}'+k)^{\overline{g}_{k}^{t}}\prod_{\substack{1\leq r\leq d\\k\in \mathbb{Z}}}(-v_{s_{n}'}-v_{r}-k_{n}'-k)^{g_{k}^{r}}
\end{equation*}
if $\mathfrak{t}_{n}$ is obtained from $\mathfrak{t}_{n-1}$ by removing a box $((i_{n}',j_{n}'),s_{n}')$, where $k_{n}'=j_{n}'-i_{n}'.$

The following proposition is inspired by [IM, Proposition 3.3] and can be proved similarly.
\begin{proposition}\label{special-propocde}
If $\bm{\lambda}$ is a $d$-partition of $n$ and $\mathfrak{t}=(\mathfrak{t}_{1},\ldots,\mathfrak{t}_{n})$ is an $n$-updown $\bm{\lambda}$-tableau, then $p_1,\ldots,p_{n}$ are all equal to zero, and $g(\mathfrak{t})$ is exactly equal to $\Theta_{\bm{\lambda}}(Q)^{-1}$ defined in $[\emph{ZL}, (3.2)]$ when $d=m$ and $v_{s}=q_{s}$ for $1\leq s\leq m.$
\end{proposition}

\subsection{Idempotents of $\mathscr{W}_{d, n}$}
Following [AMR, Definition 4.3], we say that $\mathscr{W}_{d, n}$ is generic if the parameters $v_i$, $1\leq i\leq d$, satisfy the conditions (1) the characteristic $p$ of $\mathbb{K}$ satisfies $p=0$ or $p> 2n;$ (2) $|r|\geq 2n$ whenever there exists $r\in \mathbb{Z}$ such that either $v_{i}\pm v_{j}=r$ and $i\neq j,$ or $2v_{i}=r.$ Following [Go3, Definition 4.2], we say that $\mathscr{W}_{d, n}$ is admissible if the set $\{E_{1}, E_{1}X_{1},\ldots,E_{1}X_{1}^{d-1}\}$ is linearly independent in $\mathscr{B}_{d, 2}.$ It has been proved by Goodman [Go3, Theorem 5.2] that this admissible condition coincides with the $\bm{\mathrm{u}}$-admissible condition defined in [AMR, Definition 3.6].

From now on, we always assume that $\mathscr{W}_{d, n}$ is generic and admissible. Thus, by [AMR, Lemma 4.4], we have $\mathfrak{s}=\mathfrak{t}$ if and only if $\mathrm{c}(\mathfrak{s}|k)=\mathrm{c}(\mathfrak{t}|k)$ for all $1\leq k\leq n.$ Therefore, the set $\{X_1,\ldots,X_n\},$ as a family of JM-elements for $\mathscr{W}_{d, n}$ in the abstract sense defined in [Ma, Definition 2.4], satisfies the separation condition associated to the cellular basis of $\mathscr{W}_{d, n}$ constructed in [AMR, Theorem 7.17]. In particular, we can construct the primitive idempotents of $\mathscr{W}_{d, n}$ following the arguments in [Ma, Section 3].

For each $1\leq k\leq n,$ we define the following set:
\[\mathscr{R}(k) :=\{\mathrm{c}(\mathfrak{s}|k)\:|\:\mathfrak{s}\in \mathscr{T}_{n}^{ud}(\bm{\lambda})
\text{ for some }(f, \bm{\lambda})\in \Lambda_{d,n}^{+}\}.\]
Suppose that $(f, \bm{\lambda})\in \Lambda_{d,n}^{+}$ and $\mathfrak{t}\in \mathscr{T}_{n}^{ud}(\bm{\lambda}).$ We set
\begin{equation}\label{hooklength-idempotentelement11cde}
E_{\mathfrak{t}}=\prod_{k=1}^{n}\bigg(\prod_{\substack{a\in \mathscr{R}(k)\\a\neq \mathrm{c}(\mathfrak{t}|k)}}\frac{X_{k}-a}{\mathrm{c}(\mathfrak{t}|k)-a}
\bigg).
\end{equation}
By standard arguments in [Ma, Section 3], the elements $\{E_{\mathfrak{t}}\:|\:\mathfrak{t}\in \mathscr{T}_{n}^{ud}(\bm{\lambda})
\text{ for some }(f, \bm{\lambda})\in \Lambda_{d,n}^{+}\}$ form a complete set of pairwise orthogonal primitive idempotents of $\mathscr{W}_{d, n}.$ Moreover, the elements $X_1,\ldots,X_n$ generate a maximal commutative subalgebra of $\mathscr{W}_{d, n}.$ We also have
\begin{equation}\label{hooklength-idempotentelement1111cde}
X_{k}E_{\mathfrak{t}}=E_{\mathfrak{t}}X_{k}=\mathrm{c}(\mathfrak{t}|k)E_{\mathfrak{t}}.
\end{equation}

\subsection{Fusion procedure for cyclotomic Nazarov-Wenzl algebras}
Assume that $(f, \bm{\lambda})\in \Lambda_{d,n}^{+}$ and that $\mathfrak{t}=(\mathfrak{t}_{1},\ldots,\mathfrak{t}_{n})$ is an $n$-updown $\bm{\lambda}$-tableau. Set $\bm{\mu}=\mathfrak{t}_{n-1}$ and $\mathfrak{u}=(\mathfrak{t}_{1},\ldots,\mathfrak{t}_{n-1})$ as an updown $\bm{\mu}$-tableau. Let $\bm{\theta}$ be the box that is addable to or removable from $\bm{\mu}$ to get $\bm{\lambda}.$ For simplicity, we set $\mathrm{c}_{k} :=\mathrm{c}(\mathfrak{t}|k).$ By \eqref{hooklength-idempotentelement11cde}, we can rewrite $E_{\mathfrak{t}}$ inductively as follows:
\begin{equation}\label{idempotentele-induccde}
E_{\mathfrak{t}}=E_{\mathfrak{t}}\frac{(X_{n}-a_1)\cdots (X_{n}-a_k)}{(\mathrm{c}_{n}-a_1)\cdots (\mathrm{c}_{n}-a_k)},
\end{equation}
where $a_1,\ldots,a_k$ are the contents of all boxes except $\bm{\theta},$ which can be addable to or removable from $\bm{\mu}$ to get a $d$-partition.

We denote by $\{\Delta_{1},\ldots,\Delta_{e}\}$ the set of all $d$-partitions obtained from $\bm{\mu}$ by adding a box or removing one. Set $\mathscr{S}_{j} :=(\mathfrak{t}_{1},\ldots,\mathfrak{t}_{n-1},\Delta_{j})$ for $1\leq j\leq e.$ Note that $\mathfrak{t}\in \{\mathscr{S}_{1},\ldots,\mathscr{S}_{e}\}.$ Since $\mathscr{W}_{d, n}$ is generic, hence it is semisimple. By [AMR, Theorem 5.3 a)] we have
\begin{equation}\label{sum-formula11cde}
E_{\mathfrak{u}}=\sum_{j=1}^{e}E_{\mathscr{S}_{j}}.
\end{equation}
The equality \eqref{hooklength-idempotentelement1111cde} implies that the following rational function
\begin{equation}\label{rational-function11cde}
E_{\mathfrak{u}}\frac{u-\text{c}_n}{u-X_{n}}
\end{equation}
is regular at $u=\text{c}_n,$ and by \eqref{sum-formula11cde}, we get
\begin{equation}\label{sum-function1111cde}
E_{\mathfrak{u}}\frac{u-\text{c}_n}{u-X_{n}}\Big|_{u=\text{c}_n}=E_{\mathfrak{t}}.
\end{equation}

For $1\leq i\leq n-1,$ we set
\begin{align}\label{Q-function41cde}
R_{i}(u,v;c) :=S_{i}+\frac{1}{u+v+c}-\frac{1}{u+v}E_{i}.
\end{align}
Let $\varphi_{1}(u) :=\frac{u+X_{1}+c}{u-X_1}.$ For $k=2,\ldots,n$, we set
\begin{align}\label{phi-function42cde}
\varphi_k(u_1,\ldots,u_{k-1},u)& :=R_{k-1}(u_{k-1},u;c)\varphi_{k-1}(u_1,\ldots,u_{k-2},u)S_{k-1}(u_{k-1}, u)\notag\\
=R_{k-1}&(u_{k-1},u;c)\cdots R_{1}(u_{1},u;c)\varphi_{1}(u)S_{1}(u_{1}, u)\cdots S_{k-1}(u_{k-1}, u).
\end{align}

From now on, we always set $c :=1-\frac{\omega_{0}}{2}.$ The following lemma is inspired by [IMOg2, Lemma 1] and can be proved similarly.
\begin{lemma}\label{phi-phi-phi111cde}
Assume that $n\geq 1.$ We have
\begin{align}\label{F-PhiEu43cde}
E_{\mathfrak{u}}\varphi_n(\mathrm{c}_1,\ldots,\mathrm{c}_{n-1},u)\prod_{r=1}^{n-1}g(u, \mathrm{c}_{r})^{-1}=E_{\mathfrak{u}}\frac{u+X_{n}+c}{u-X_n}.
\end{align}
\end{lemma}
\begin{proof}
We shall prove \eqref{F-PhiEu43cde} by induction on $n.$ For $n=1,$ the situation is trivial.

We set
\begin{align}\label{phi-function421cde}
\varphi'_n(\mathrm{c}_1,&\ldots,\mathrm{c}_{n-1},u)\notag\\
&=R_{n-1}(\mathrm{c}_{n-1},u;c)\cdots R_{1}(\mathrm{c}_{1},u;c)\varphi_{1}(u)S_{1}(u, \mathrm{c}_{1})^{-1}\cdots S_{n-1}(u, \mathrm{c}_{n-1})^{-1}.
\end{align}
By \eqref{Baxterized-elements111cde} and \eqref{phi-function421cde}, in order to show \eqref{F-PhiEu43cde}, it suffices to prove that
\begin{align}\label{F-PhiEu4321cde}
E_{\mathfrak{u}}\varphi'_n(\mathrm{c}_1,\ldots,\mathrm{c}_{n-1},u)=E_{\mathfrak{u}}\frac{u+X_{n}+c}{u-X_n}.
\end{align}
By the induction hypothesis, it boils down to proving the following equality:
\begin{align}\label{EUEU-PhiEu5cde}
E_{\mathfrak{u}}R_{n-1}(\mathrm{c}_{n-1},u;c)\frac{u+X_{n-1}+c}{u-X_{n-1}}S_{n-1}(u, \mathrm{c}_{n-1})^{-1}=E_{\mathfrak{u}}\frac{u+X_{n}+c}{u-X_n}.
\end{align}
Since $X_{n}$ commutes with $E_{\mathfrak{u}},$ we can rewrite \eqref{EUEU-PhiEu5cde} as follows:
\begin{align}\label{EUEU-PhiEu6cde}
E_{\mathfrak{u}}(u-X_n)R_{n-1}(&\mathrm{c}_{n-1},u;c)(u+X_{n-1}+c)\notag\\
&=E_{\mathfrak{u}}(u+X_{n}+c)S_{n-1}(u, \mathrm{c}_{n-1})(u-X_{n-1}).
\end{align}
By \eqref{Baxterized-elements11cde} and \eqref{Q-function41cde}, the equality \eqref{EUEU-PhiEu6cde} becomes
\begin{align}\label{EUEU-PhiEu7cde}
E_{\mathfrak{u}}&(u-X_n)\Big(S_{n-1}+\frac{1}{\mathrm{c}_{n-1}+u+c}-\frac{1}{\mathrm{c}_{n-1}+u}E_{n-1}\Big)(u+X_{n-1}+c)\notag\\
&=E_{\mathfrak{u}}(u+X_{n}+c)\Big(S_{n-1}+\frac{1}{\mathrm{c}_{n-1}-u}-\frac{1}{\mathrm{c}_{n-1}-u+\frac{\omega_{0}}{2}-1}E_{n-1}\Big)(u-X_{n-1}).
\end{align}

By definition, we have $S_{n-1}X_{n-1}=X_{n}S_{n-1}+E_{n-1}-1.$ Thus, we get that \eqref{EUEU-PhiEu7cde} is equivalent to
\begin{align}\label{EUEU-PhiEu8cde}
E_{\mathfrak{u}}(u&-X_n)\Big(uS_{n-1}+(X_{n}S_{n-1}+E_{n-1}-1)+cS_{n-1}+1\notag\\
&\hspace{2cm}-\frac{1}{\mathrm{c}_{n-1}+u}E_{n-1}(u+X_{n-1}+c)\Big)\notag\\
&=E_{\mathfrak{u}}(u+X_{n}+c)\Big(uS_{n-1}-(X_{n}S_{n-1}+E_{n-1}-1)-1\notag\\
&\hspace{2cm}-\frac{1}{\mathrm{c}_{n-1}-u+\frac{\omega_{0}}{2}-1}E_{n-1}(u-X_{n-1})\Big).
\end{align}
It is easy to see that the equality \eqref{EUEU-PhiEu8cde} comes down to the following equality:
\begin{align}\label{EUEU-PhiEu9cde}
(c+2u)&E_{\mathfrak{u}}-E_{\mathfrak{u}}(u-X_n)\frac{1}{\mathrm{c}_{n-1}+u}E_{n-1}(u+X_{n-1}+c)\notag\\
&=-E_{\mathfrak{u}}(u+X_{n}+c)\frac{1}{\mathrm{c}_{n-1}-u+\frac{\omega_{0}}{2}-1}E_{n-1}(u-X_{n-1}).
\end{align}

By definition, we have $E_{\mathfrak{u}}X_{n-1}=\mathrm{c}_{n-1}E_{\mathfrak{u}}.$ Hence, we get $E_{\mathfrak{u}}X_{n}E_{n-1}=-\mathrm{c}_{n-1}E_{\mathfrak{u}}E_{n-1}$ by definition.
According to this, by comparing the coefficients of the terms involving $E_{\mathfrak{u}}E_{n-1}X_{n-1}$, we see that it suffices to show that
\begin{align}\label{EUEU-PhiEu11cde}
\frac{-u-\mathrm{c}_{n-1}}{\mathrm{c}_{n-1}+u}=\frac{u-\mathrm{c}_{n-1}+c}{\mathrm{c}_{n-1}-u+\frac{\omega_{0}}{2}-1}.
\end{align}
By comparing the coefficients of the terms involving $E_{\mathfrak{u}}E_{n-1}$, it suffices to show that
\begin{align}\label{EUEU-PhiEu12cde}
(c+2u)+\frac{-(c+u)(\mathrm{c}_{n-1}+u)}{\mathrm{c}_{n-1}+u}=\frac{u(-u+\mathrm{c}_{n-1}-c)}{\mathrm{c}_{n-1}-u+\frac{\omega_{0}}{2}-1}.
\end{align}
Noting that $c=1-\frac{\omega_{0}}{2},$ it is easy to verify that \eqref{EUEU-PhiEu11cde} and \eqref{EUEU-PhiEu12cde} are true. Thus, \eqref{EUEU-PhiEu9cde} holds. The lemma is proved.
\end{proof}

Let $\overline{\varphi}_{1}(u) :=(u-v_{1})\cdots (u-v_{d})\frac{u+X_{1}+c}{u-X_1}.$ For $k=2,\ldots,n$, we set
\begin{align}\label{phi-function424242cde}
\overline{\varphi}_k(u_1,\ldots,u_{k-1},u)& :=R_{k-1}(u_{k-1},u;c)\overline{\varphi}_{k-1}(u_1,\ldots,u_{k-2},u)S_{k-1}(u_{k-1}, u)\notag\\
=R_{k-1}&(u_{k-1},u;c)\cdots R_{1}(u_{1},u;c)\overline{\varphi}_{1}(u)S_{1}(u_{1}, u)\cdots S_{k-1}(u_{k-1}, u).
\end{align}

We also define the following rational function:
\begin{align}\label{Phi-function111cde}
\Psi(u_1,\ldots,u_n) :=\overline{\varphi}_1(u_1)\cdots \overline{\varphi}_{n-1}(u_1,\ldots,u_{n-1})\overline{\varphi}_n(u_1,\ldots,u_{n}).
\end{align}
Recall that the integers $p_{1},\ldots,p_{n}$ associated to $\mathfrak{t}$ have been defined as in \eqref{integer-indexbar} or \eqref{integer-indexbar11}.

Now we can state the main result of this paper.

\begin{theorem}\label{main-theorem11112cde}
The idempotent $E_{\mathfrak{t}}$ of $\mathscr{W}_{d, n}$ corresponding to an $n$-updown $\bm{\lambda}$-tableau $\mathfrak{t}$ can be derived by the following consecutive evaluations$:$
\begin{equation}\label{idempotents111cde}
E_{\mathfrak{t}}=\frac{1}{g(\mathfrak{t})}\Big(\prod_{k=1}^{n}\frac{(u_{k}-\mathrm{c}_{k})^{p_{k}}}{u_{k}+\mathrm{c}_{k}+c}\Big)
\Psi(u_1,\ldots,u_n)\Big|_{u_{1}=\emph{c}_1}\cdots\Big|_{u_{n}=\emph{c}_{n}}.
\end{equation}
\end{theorem}
\begin{proof}
We shall prove the theorem by induction on $n.$ For $n=1,$ we have $p_{1}=0$ by Proposition \ref{special-propocde}. Thus, we get that the right-hand side of \eqref{idempotents111cde} is equal to
\begin{align}\label{n-1-istruecde}
\frac{1}{g(\mathfrak{t})}&\frac{(u_{1}-v_{1})\cdots (u_{1}-v_{d})}{u_{1}+\mathrm{c}_{1}+c}
\frac{u_{1}+X_{1}+c}{u_{1}-X_1}\Big|_{u_{1}=\mathrm{c}_1}\notag\\
&=\frac{1}{g(\mathfrak{t})}\frac{(u_{1}-v_{1})\cdots (u_{1}-v_{d})}{u_{1}-\mathrm{c}_{1}}\frac{u_{1}-\mathrm{c}_{1}}{u_{1}+\mathrm{c}_{1}+c}\frac{u_{1}+X_{1}+c}{u_{1}-X_1}\Big|_{u_{1}=\mathrm{c}_1}.
\end{align}
Moreover, by \eqref{hooklength-indexbar11cde}, we have \[g(\mathfrak{t})=\prod_{1\leq k\leq d;v_{k}\neq \mathrm{c}_{1}}(\mathrm{c}_{1}-v_{k}).\]
Therefore, it is easy to see that \eqref{n-1-istruecde} is equal to $E_{\mathfrak{t}}$ by \eqref{hooklength-idempotentelement1111cde} and \eqref{sum-function1111cde}.

For $n\geq 2,$ by the induction hypothesis we can write the right-hand side of \eqref{idempotents111cde} as follows:
\begin{align}\label{n-1-istrue2cde}
\frac{g(\mathfrak{u})}{g(\mathfrak{t})}\frac{(u_{n}-\mathrm{c}_{n})^{p_{n}}}{u_{n}+\mathrm{c}_{n}+c}E_{\mathfrak{u}}
\overline{\varphi}_n(\mathrm{c}_{1},\ldots,\mathrm{c}_{n-1},u_n)\Big|_{u_{n}=\mathrm{c}_{n}}.
\end{align}
Note that $\overline{\varphi}_n(\mathrm{c}_{1},\ldots,\mathrm{c}_{n-1},u_n)=(u_{n}-v_{1})\cdots (u_{n}-v_{d})\varphi_n(\mathrm{c}_{1},\ldots,\mathrm{c}_{n-1},u_n).$ By \eqref{F-PhiEu43cde}, we can rewrite the expression \eqref{n-1-istrue2cde} as
\begin{align}\label{n-1-istrue3cde}
\frac{g(\mathfrak{u})}{g(\mathfrak{t})}\frac{(u_{n}-\mathrm{c}_{n})^{p_{n}}}{u_{n}+\mathrm{c}_{n}+c}(u_{n}-v_{1})\cdots (u_{n}-v_{d})\prod_{r=1}^{n-1}g(u_{n}, \mathrm{c}_{r})E_{\mathfrak{u}}\frac{u_{n}+X_{n}+c}{u_{n}-X_n}\Big|_{u_{n}=\mathrm{c}_{n}}.
\end{align}

By \eqref{hooklength-indexbar11cde}, we see that
\begin{align*}
\frac{g(\mathfrak{u})}{g(\mathfrak{t})}(u_{n}&-v_{1})\cdots (u_{n}-v_{d})\prod_{r=1}^{n-1}g(u_{n}, \mathrm{c}_{r})(u_{n}-\mathrm{c}_{n})^{p_{n}-1}\notag\\
&=\frac{g(\mathfrak{u})}{g(\mathfrak{t})}(u_{n}-v_{1})\cdots (u_{n}-v_{d})\prod_{r=1}^{n-1}\frac{(u_{n}-\mathrm{c}_{r}+1)(u_{n}-\mathrm{c}_{r}-1)}{(u_{n}-\mathrm{c}_{r})^{2}}(u_{n}-\mathrm{c}_{n})^{p_{n}-1}
\end{align*}
is regular at $u_n=\mathrm{c}_{n}$ and is equal to $1.$ Thus, the expression \eqref{n-1-istrue3cde} equals
\begin{align}\label{n-1-istrue4cde}
E_{\mathfrak{u}}\frac{u_{n}-\mathrm{c}_{n}}{u_{n}-X_n}\frac{u_{n}+X_{n}+c}{u_{n}+\mathrm{c}_{n}+c}\Big|_{u_{n}=\mathrm{c}_{n}}.
\end{align}
By \eqref{sum-function1111cde}, we see that \eqref{n-1-istrue4cde} is equal to
\begin{align}\label{n-1-istrue5cde}
E_{\mathfrak{t}}\frac{u_{n}+X_{n}+c}{u_{n}+\mathrm{c}_{n}+c}\Big|_{u_{n}=\mathrm{c}_{n}}.
\end{align}
By \eqref{hooklength-idempotentelement1111cde}, we have $E_{\mathfrak{t}}X_{n}=\mathrm{c}_{n}E_{\mathfrak{t}}.$ Thus, we get that the expression \eqref{n-1-istrue5cde}, that is, the right-hand side of \eqref{idempotents111cde} equals $E_{\mathfrak{t}}.$
 \end{proof}

\begin{remark}\label{remark111cde}
Let $\mathscr{D}_{d, n}$ be the degenerate cyclotomic Hecke algebra. It has been proved in [AMR, Proposition 7.2] that $\mathscr{D}_{d, n}$ is isomorphic to the quotient of $\mathscr{W}_{d, n}$ by the two-sided ideal generated by all $E_{i}.$ In the process of taking quotient, the parameter $\omega_{0}$ disappears; however, the parameter $c$ is reserved and can be arbitrary. If we replace the $S_{i}(u, v),$ $R_{i}(u,v;c),$ $\varphi_{1}(u)$ in \eqref{phi-function42cde} with
\begin{align*}
\overline{S}_{i}(u,v)=S_{i}+\frac{1}{v-u},\quad \overline{R}_{i}(u,v;c) :=S_{i}+\frac{1}{u+v+c},\quad \chi_{1}(u) :=\frac{u+X_{1}+c}{u-X_1},
\end{align*}
it is easy to see that the analogue of Lemma \ref{phi-phi-phi111cde} holds.

Let $\overline{\chi}_{1}(u) :=(u-v_{1})\cdots (u-v_{d})\frac{u+X_{1}+c}{u-X_1},$ and for $k=2,\ldots,n$, set
\begin{align*}
\overline{\chi}_k(u_1,\ldots,u_{k-1},u)& :=\overline{R}_{k-1}(u_{k-1},u;c)\overline{\chi}_{k-1}(u_1,\ldots,u_{k-2},u)\overline{S}_{k-1}(u_{k-1}, u).
\end{align*}

We also define a rational function by
\begin{align*}
\Omega(u_1,\ldots,u_n) :=\overline{\chi}_1(u_1)\cdots \overline{\chi}_{n-1}(u_1,\ldots,u_{n-1})\overline{\chi}_n(u_1,\ldots,u_{n}).
\end{align*}
Then it is easy to see that the analogue of Theorem \ref{main-theorem11112cde} is true. Thus, we get a one-parameter family of the fusion procedures for degenerate cyclotomic Hecke algebras, generalizing the results obtained in [ZL].
\end{remark}

\noindent{\bf Acknowledgements.}
The author is deeply indebted to Dr. Shoumin Liu for posing the question about fusion procedures for cyclotomic Nazarov-Wenzl algebras to him.



\end{document}